\documentclass[12pt]{amsart}
\usepackage{amsmath,amsthm,amsfonts,amssymb,xcolor,comment,hyperref,amsaddr}
\usepackage{graphicx}
\usepackage{setspace}
\usepackage{mathrsfs}
\usepackage{hyperref}
\usepackage[letterpaper,margin=1.3in]{geometry}
\usepackage[shortlabels]{enumitem}

\newcommand{\scott}[1]{{\scriptsize\ \color{orange}\textbf{Scott's note:} #1 \color{black}\ \normalsize}}

\newcommand{\R}{\mathbb{R}}

\renewcommand{\epsilon}{\varepsilon}
\renewcommand{\phi}{\varphi}
\renewcommand{\hat}{\widehat}
\renewcommand{\tilde}{\widetilde}

\newcommand{\supp}{\mathbf{supp}}

\newcommand{\reg}{\text{reg}}

\newtheorem{theorem}{Theorem}[section]
\newtheorem{proposition}[theorem]{Proposition}
\newtheorem{lemma}[theorem]{Lemma}
\newtheorem{corollary}[theorem]{Corollary}

\theoremstyle{definition}
\newtheorem{definition}[theorem]{Definition}
\newtheorem*{question}{Question}

\theoremstyle{remark}
\newtheorem{remark}[theorem]{Remark}

\numberwithin{equation}{section}

\newcommand{\diam}{\mathop\mathrm{diam}\nolimits}

\author{Scott Zimmerman}
\title{Singular integrals on $C_{w^*}^{1,\alpha}$ regular curves in Banach duals}

\allowdisplaybreaks

\begin{document}

\maketitle

\begin{abstract}
The modern study of singular integral operators on curves in the plane began in the 1970's. Since then, there has been a vast array of work done on the boundedness of singular integral operators defined on lower dimensional sets in Euclidean spaces. In recent years, mathematicians have attempted to push these results into a more general metric setting particularly in the case of singular integral operators defined on curves and graphs in Carnot groups.

Suppose $X = Y^*$ for a separable Banach space $Y$. Any separable metric space can be isometrically embedded in such a Banach space via the Kuratowski embedding. Suppose $\Gamma = \gamma([a,b])$ is a curve in $X$ whose $w^*$-derivative is H\"{o}lder continuous and bounded away from 0.
We prove that any convolution type singular integral operator associated with a 1-dimensional Calder\'{o}n-Zygmund kernel which is uniformly $L^2$-bounded on lines is $L^p$-bounded along $\Gamma$.

We also prove a version of David's ``good lambda'' theorem for upper regular measures on doubling metric spaces.
\end{abstract}

\tableofcontents

\section{Introduction}

This paper will be concerned with singular integral operators (SIOs) generated by 1-dimensional Calder\'{o}n-Zygmund kernels in the non-Euclidean setting.
For all relevant definitions, see Section~\ref{sec:SIO}.

The modern study of SIOs on curves
began with the works of Calder\'{o}n \cite{Calderon} and 
Coifman, McIntosh, and Meyer \cite{CMM}
on the $L^2$ boundedness of the Cauchy integral operator in the complex plane.
David \cite{David1984} 
and
Mattila, Melnikov, and Verdera \cite{MatMelVer}
classified the regular sets $E$ 
for which the Cauchy operator is bounded in $L^2(E)$ as those sets which are contained in a regular curve.
A curve $\Gamma$ in a metric space $X$ is {\em regular} if the length of the curve within a ball is comparable to the radius of the ball. See \eqref{d-regular}.
A thorough discussion of bounded SIOs on lower dimensional sets and their applications can be found in Tolsa's 2014 book \cite{tolsabook}.

Work on low dimensional SIOs in the non-Euclidean setting
has attracted recent interest.
In particular, there has been progress regarding the properties of SIOs on lower dimensional sets in Carnot groups \cite{CFOsios,ChoFasOrp,ChoLi,SIOCLZ,ChoMat,FasOrpReg}.
For example, any kernel in the Heisenberg group which is uniformly $L^2$-bounded on vertical planes corresponds to a SIO which is bounded on any 3-regular $C^{1,\alpha}$ intrinsic Lipschitz graph with compact support \cite{CFOsios}.
Further, any kernel in an arbitrary Carnot group which is uniformly $L^2$-bounded on (horizontal) lines (a property which we will call UBL)
corresponds to a SIO which is bounded along any $C^{1,\alpha}$ regular curve \cite{SIOCLZ}.
Recently, 
F\"{a}ssler and Orponen \cite{FasOrpReg}
questioned if every UBL kernel in the Heisenberg group generates a SIO which is bounded along regular curves, and they provided some evidence toward a positive answer.

Ideally, one would like to 
fully understand the behavior of SIOs and their relationship to the geometry of a curve in a general metric setting.
As a first step,
we will consider SIOs generated by kernels 
in any Banach space $X$ which is the dual of a separable Banach space $Y$ i.e. $X = Y^*$.
The inspiration for this decision was the famous Kuratowski embedding \cite{Kurat}:
any separable metric space $M$ can be embedded isometrically into $X = \ell^\infty$,
and $\ell^\infty$ is the dual of the separable Banach space $Y = \ell^1$.
In particular, we may embed a curve in any metric space isometrically into $\ell^\infty$.
Therefore, an exploration of the theory of SIOs on curves in such Banach spaces is a reasonable first step toward an understanding of
SIOs in arbitrary metric spaces.

This paper will focus on the main result from \cite{SIOCLZ} mentioned above
(which addressed SIOs along $C^{1,\alpha}$ regular curves in Carnot groups)
in the Banach space setting.
The proofs in \cite{SIOCLZ} strongly relied on the smoothness of the underlying curve.
As our end goal is the study of SIOs in arbitrary metric spaces,
it seems natural to consider curves which are smooth with respect to 
the {\em metric derivative} defined as follows for a curve $\gamma:(a,b) \to X$: 
$$
|\dot{\gamma}|(x) = \lim_{h \to 0} \frac{d(\gamma(x+h),\gamma(x))}{|h|}.
$$
Unfortunately for us, 
the metric derivative computes only the speed of a curve.
It provides no directional information,
and such information is essential for the understanding of SIOs via linear approximations.

Another notion of differentiability
which encodes both speed and direction will be necessary
for the purposes of this paper.
Suppose $\Omega \subset \mathbb{R}^n$ is open and $X = Y^*$ for a separable Banach space $Y$.
Following the examples of \cite{AmbKir,HeiMan},
we say that a map 
$f:\Omega \to X$ is {\em $w^*$-differentiable} at $x \in \Omega$
if there is a linear map $D^*f_x:\mathbb{R}^n \to X$ such that
$$
\lim_{z \to x}
\left\langle \frac{f(z)-f(x)-D^*f_x(z-x)}{|z-x|},y \right\rangle = 0
\quad 
\text{for all } y \in Y.
$$

In the language of curves, 
this means that $\gamma:[a,b] \to X$
is {\em $w^*$-differentiable} at $t \in [a,b]$
if there is a point $\gamma'(t) \in X$ such that
$$
\tfrac{d}{ds} \langle \gamma(s),y \rangle|_{s=t} = \langle \gamma'(t),y \rangle
\quad
\text{ for all } y \in Y.
$$
We say that $\gamma$ is $C_{w^*}^{1,\alpha}$ 
if $\gamma'(t)$ exists for every $t \in [a,b]$ and $\gamma'$ is $\alpha$-H\"{o}lder continuous.

With this notion of differentiability in hand, 
we are now ready to state the main result of the paper.

\begin{theorem}
\label{t-main}
Fix $\alpha \in (0,1]$ and $p \in (1,\infty)$.
Suppose $X$ is the dual of a separable Banach space
and $\Gamma$ is a $C_{w^*}^{1,\alpha}$ regular curve in $X$.
Then any 
convolution type
 singular integral operator corresponding to a 1-dimensional CZ kernel which is uniformly $L^2$-bounded on lines
 is $L^p(\mathcal{H}^1|_E)$-bounded for any regular set $E \subset \Gamma$.
\end{theorem}


The statement and proof of this result are 
largely similar to those of the main theorem in \cite{SIOCLZ}.
The main differences (which will be addressed in this paper) are the preliminary properties of $w^*$-smooth curves (as described in Section~\ref{s-weakstar}) and the ``good lambda'' result (Theorem~\ref{t:goodlambda}).  Moreover, the examples of CZ-kernels in Section~\ref{sec:SIO} are original in the Banach space setting.
Note that Theorem~\ref{t-main} is not advertised as a generalization of the main result of \cite{SIOCLZ}. I am not currently aware of an embedding of a curve in a Carnot group into $\ell^{\infty}$ through which the result in \cite{SIOCLZ} follows from Theorem~\ref{t-main}.
Rather, Theorem~\ref{t-main} is simply of independent interest in the Banach space setting.

Note that, since we are considering Banach spaces, one may suppose that a more appropriate choice of derivative is the 
Fr\'{e}chet derivative. 
However, Theorem~\ref{t-main} would in fact be weakened if $\Gamma$ was assumed to be $C^{1,\alpha}$ with respect to the Fr\'{e}chet derivative rather than the $w^*$-derivative. 
Indeed, if the Fr\'{e}chet derivative $D \gamma_x$ of a curve $\gamma$ exists at a point $x$, then $\gamma$ is also $w^*$-differentiable at $x$ with $\gamma'(x) = D\gamma_x$.
These notions of differentiability are distinct in general.
For example, as noted in the introduction of \cite{AmbKir}, the curve $t \mapsto \chi_{(0,t)}$
in $(C[0,1])^*$
is $w^*$-differentiable 
with $w^*$-derivative $\delta_t$,
but it is not Fr\'{e}chet differentiable.

An important (and standard) part of the proof of Theorem~\ref{t-main} 
is the ``big pieces'' theorem of David \cite{MR956767}
which utilizes a standard ``good $\lambda$'' inequality. 
The theorem allows us reduce our attention from the entire set $\Gamma$ to a suitably large subset of any ball.
The original result of David was stated for regular measures in $\mathbb{R}^n$.
Section~\ref{s-bigpiece} below
contains a general statement and proof of a ``good lambda'' result
for upper regular measures (which are not necessarily doubling) on doubling metric spaces
similar to a theorem of Tolsa \cite[Theorem 2.22]{tolsabook}.
See Theorem~\ref{t:goodlambda}.

The paper is organized as follows.
In Section~\ref{sec:SIO},
we will see the definitions of Calder\'{o}n-Zygmund (CZ) kernels, singular integral operators, and their boundedness.
We will also see examples of CZ kernels in our Banach space setting.
In Section~\ref{s-weakstar}, we will explore the notion of $w^*$-differentiability along curves and establish important properties of $w^*$-smooth curves.
An outline of the proof of Theorem~\ref{t-main} is presented in Section~\ref{sec-proof}.
Finally, as discussed above, the proof of a ``good lambda'' result for upper regular measures on doubling metric spaces is given in Section~\ref{s-bigpiece}.

\section{Singular integral operators in Banach spaces}
\label{sec:SIO}

Suppose $M$ is a metric space with metric $d$.
For $x \in M$ and $r > 0$,
we will use $B(x,r)$ to denote the open ball
$$
B(x,r) := \{ y \in M : d(x,y) < r \}.
$$
A Radon measure $\mu$ on $M$ is called
1-{\em Ahlfors-David-regular} (or simply {\em regular})
if there is a constant $C \geq 1$ such that
\begin{equation}
    \label{d-regular}
C^{-1}r \leq \mu(B(x,r)) \leq Cr
\quad
\text{for all } x \in X
\text{ and }
0 < r \leq \diam(\text{supp}(\mu)).
\end{equation}
If only the upper inequality holds, we say that $\mu$ is {\em upper regular}. 
A closed set $E \subset X$ 
is called {\em regular}
if 
the restriction of the 1-dimensional Hausdorff measure to $E$ (denoted $\mathcal{H}^1|_E$) is regular.
We will denote the constant $C$
associated with a measure $\mu$ or set $E$ by $\reg(\mu)$ or $\reg(E)$
respectively.

An {\em n-dimensional Calder\'{o}n-Zygmund (CZ) kernel} 
$K$ on a metric space $M$ is a continuous function
$K:(M \times M) \setminus \Delta \to \mathbb{R}$
where $\Delta = \{ (x,x) \, : \, x \in M \}$
satisfying the following for some $C>0$ and $\beta \in (0,1]$:
\begin{equation}
\label{e-growth-metric}
|K(x,y)| \leq \frac{C}{d(x,y)^n}
\end{equation}
and
\begin{equation}
\label{e-growth-holder}
|K(x,y)-K(x',y)| 
+ 
|K(y,x) - K(y,x')|
\leq C\frac{d(x,x')^\beta}{d(x,y)^{n+\beta}}
\end{equation}
for all $x,x',y \in M$
with
$d(x,x') \leq d(x,y)/2$.

\subsection{Singular integral operators}
Suppose $K$ is a 1-dimensional CZ kernel 
and $\mu$ is an upper regular measure on a metric space $M$. 
For each $\varepsilon>0$, define the {\em truncated singular integral operator} $T_{\mu,\varepsilon}$ {\em associated with $K$} as
$$
T_{\mu,\varepsilon}f(x)
=
\int_{\Vert x-y\Vert >\varepsilon}
K(x,y)f(y)\, d\mu(y)
\quad
\text{for any } f \in L^p(\mu), \, 1<p<\infty.
$$
The associated {\em maximal singular integral} $T_{\mu,*}$ is defined as
$$
T_{\mu,*} f(x) = \sup_{\varepsilon > 0} \left| T_{\mu,\varepsilon}f(x) \right|
\quad
\text{for any } f \in L^p(\mu), \, 1<p<\infty.
$$
If $\mu$ is further assumed to be regular,
we define the symbol
$$
\Vert T_\mu \Vert_{L^p(\mu) \to L^p(\mu)} := \inf\{ M > 0 \, : \, \Vert T_{\mu,\varepsilon}f\Vert_{L^p(\mu)}
\leq M \Vert f \Vert_{L^p(\mu)}
\text{ for }
f\in L^p(\mu), \, \varepsilon > 0\}
$$
for each $p \in (1,\infty)$.
We say that {\em the singular integral operator $T_\mu$ associated with $K$ is bounded on $L^p(\mu)$} 
if 
$\Vert T_\mu \Vert_{L^p(\mu) \to L^p(\mu)}< \infty$.
In other words, $T_\mu$ is bounded if the operators $T_{\mu,\varepsilon}$ are bounded uniformly on $L^p(\mu)$ for all $\varepsilon >0$.

\subsection{CZ kernels in Banach spaces}
Suppose $X$ is a Banach space with norm $\Vert \cdot \Vert$.
Following the above definitions, we say that
a continuous function $K:X \setminus \{0\} \to \mathbb{R}$
is a 1-dimensional CZ kernel if the function
$K':(X \times X)\setminus \Delta \to \R$
defined as
$K'(x,y) = K(x-y)$ is a 1-dimensional CZ kernel.
In other words, 
there are constants $B>0$ and $\beta \in (0,1]$ satisfying the growth condition
\begin{equation}
    \label{e-growth}
    |K(x)| \leq \frac{B}{\Vert x \Vert}
    \quad
    \text{ for all } x \in X\setminus \{ 0 \},
\end{equation}
and the H\"{o}lder continuity condition 
\begin{equation}
    \label{e-holder}
|K(x) - K(x + h)| 
\leq 
B\frac{\Vert h\Vert^\beta}{\Vert x\Vert^{1+\beta}}
\quad \text{ for all } x,h \in X \text{ with } \Vert h \Vert \leq \tfrac12 \Vert x \Vert.
\end{equation}
We define truncated and maximal singular integral operators $T_{\mu,\varepsilon}$ and $T_{\mu,*}$ associated with $K$ to be the respective operators associated with $K'$,
and we say that the singular integral operator $T_\mu$ associated with $K$ is bounded on $L^p(\mu)$
if the operators $T_{\mu,\varepsilon}$ associated with $K'$ are bounded uniformly on $L^p(\mu)$ for all $\varepsilon >0$.


\subsection{Examples of CZ kernels in Banach spaces} \label{ss-examples}
{\em (1)} Suppose $X$ is a Banach space with a $C^1$ norm $\Vert \cdot \Vert$
i.e. the norm is Fr\'{e}chet differentiable on $X \setminus \{0\}$ 
and this Fr\'{e}chet derivative is continuous.
Examples of such Banach spaces are Hilbert spaces and $L^p(\Omega)$ for any metric measure space $\Omega$ with $p > 1$ \cite{BonFram}.
Suppose also that $X$ has
a normalized Schauder basis $(e_j)_{j \in \mathbb{N}}$.

Define the {\em 1-dimensional Banach-Riesz kernels} $R_n$ on $X$ as
$$
R_n(x) = \frac{x_n}{\Vert x \Vert^2}
\quad
\text{for } x \neq 0, \, n \in \mathbb{N}
$$
where $x = \sum_{n=1}^\infty x_n e_n$.

We will now check that each $R_n$
is a 1-dimensional CZ kernel. 
It follows from the open mapping theorem that the basis projections $P_n(x) := \sum_{j=1}^n x_je_j$ are uniformly bounded by a constant $C > 0$
\cite[Theorem 4.10]{BanachBook}.
Thus
$$
|x_n| = \Vert x_n e_n \Vert = \Vert P_n(x) - P_{n-1}(x)\Vert \leq 2C \Vert x \Vert.
$$
That is, $R_n$ satisfies \eqref{e-growth} uniformly in $n$. 
Now, according to the mean value theorem, 
for any $z$ and $x$ with $\Vert z \Vert = 1$ and $\Vert x \Vert \leq \tfrac12$
there is a point $c$ on the segment from $z$ to $z-x$ 
and a constant $C'>0$ depending only on the norm on $X$
such that
\begin{equation}
\label{e-rieszholder}
|R_n(z) - R_n(z-x)| = \Vert DR_n(c) \Vert \cdot \Vert x \Vert \leq C' \Vert x \Vert
\end{equation}
(where $DR_n$ is the Frech\'{e}t derivative of $R_n$)
since $\tfrac12 \leq \Vert z-x \Vert \leq \tfrac32$ and $R_n$ is $C^1$ away from 0.
Since
$R_n$ is $-1$-homogeneous (i.e. $R_n(rx)=r^{-1}R_n(x)$),
the H\"{o}lder condition \eqref{e-holder} follows 
(uniformly in $n$)
from \eqref{e-rieszholder} with $\beta =1$, and hence each $R_n$ 
is a 1-dimensional CZ-kernel.

{\em (2)} Assume that $X = Y^*$ where $Y$ is separable,
and again assume that $X$ has a $C^1$ norm $\Vert \cdot \Vert$.
Say $(y_n)_{n \in \mathbb{N}}$ is a dense set in the unit sphere of $Y$.
Define the {\em 1-dimensional dual-Riesz kernels} on $X$ as
$$
R^*_n(x) = \frac{\langle x,y_n\rangle}{\Vert x \Vert^2}
\quad
\text{for } x \neq 0, \, n \in \mathbb{N}
$$
where $\langle x , \cdot \rangle$ is evaluation of the functional $x$.
According to the Cauchy-Schwartz inequality, $R_n^*$ satisfies \eqref{e-growth} with the constant $B=1$,
and, as above, we may conclude that $R_n^*$ satisfies \eqref{e-holder}.
Thus each $R_n^*$ is a CZ kernel on $X$.

\subsection{The UBL condition and annular boundedness}
Suppose again that $X$ is a Banach space.
A continuous function $K : X \setminus \{ 0 \} \to \mathbb{R}$
satisfying the growth condition \eqref{e-growth}
is said to be {\em uniformly bounded on lines} (or {\em UBL})
if there is a constant $M>0$ such that
$$
\Vert T_{\mathcal{H}^1|_L} \Vert_{L^2(\mathcal{H}^1|_L) \to L^2(\mathcal{H}^1|_L)} \leq M < \infty
$$
for any line $L$ in $X$
where $T_{\mathcal{H}^1|_L}$
is the SIO associated with $K$.

As mentioned in the introduction,
F\"{a}ssler and Orponnen \cite{FasOrpReg} asked if every 1-dimensional CZ kernel in the Heisenberg group 
which is also UBL generates a SIO which is bounded on any regular curves.
Some progress was made in \cite{SIOCLZ} wherein it was shown that such kernels are bounded on $C^{1,\alpha}$
regular curves in any Carnot group.
It therefore seems natural to study this class of kernels in any metric space which possesses some sort of affine structure.

The proof of the following proposition is identical to that of \cite[Lemma 2.9]{CFOsios}.
Indeed, any line in $X$ is a scaled isometric copy of $\mathbb{R}$
which allows us to apply the same Fourier transform arguments in $\mathbb{R}$
as seen in \cite{CFOsios}.
Fix a smooth, even function $\psi:\mathbb{R} \to \mathbb{R}$
satisfying $\chi_{B(0,1/2)} \leq \psi \leq \chi_{B(0,2)}$, 
and define for any $r>0$ the radial functions $\psi_r :X \to \mathbb{R}$ as
$$
\psi_r(x) := \psi(r\Vert x \Vert) 
\quad 
\text{for all } x \in X.
$$

\begin{proposition}
\label{p-annular}
If $K$ is a 1-dimensional CZ kernel which is UBL,
then $K$ satisfies the annular boundedness condition.
That is, 
there is a constant $A \geq 1$
such that, for any line $L$,
\begin{equation}
\label{annular}
\left| \int_L [\psi_R(y) - \psi_r(y) ] K(y) \, d\mathcal{H}^1(y) \right| \leq A
\quad \text{for all } 0<r<R.
\end{equation}
\end{proposition}
The notion of annular boundedness has appeared in \cite{CFOsios} and \cite{SIOCLZ} in the setting of Carnot groups.
In fact, we could argue as in \cite[Proposition 2.16]{SIOCLZ} to show that the UBL and annular boundedness conditions are equivalent, but this will not be needed below.

\subsection{Weak (1,1) estimates}
We will now record two results which will be very useful in the proof of Theorem~\ref{t-main}.
Suppose $M$ is a metric space with Radon measure $\mu$.
A linear operator $T$ is said to be {\em of weak type $(1,1)$} if there is a constant $C>0$ such that
\begin{equation}
    \label{d:w11}
\mu(\{x \in M : |Tf(x)| > \lambda \})
\leq
\frac{C}{\lambda} \Vert f \Vert_{L^1(\mu)}
\quad
\text{for all }
f \in L^1(\mu).
\end{equation}
Our first result follows immediately from the theorem of Nazarov, Treil, and Volberg
(see Theorem~9.1 in \cite{NazTreVol}).
\begin{lemma}
\label{t-L2toWeak}
Suppose $M$ is a separable metric space and $\mu$ is an upper regular measure on $M$.
If $K$ is a 1-dimensional CZ kernel on $M$
such that the SIO $T_\mu$ associated with $K$ is bounded in $L^2(\mu)$,
then $T_{\mu,*}$ is of weak type $(1,1)$
where the constant in the definition \eqref{d:w11} depends only on $K$ and $\Vert T_\mu \Vert_{L^2(\mu) \to L^2(\mu)}$.
\end{lemma}

For a Radon measure $\mu$ on a metric space $M$,
we define the maximal operator $M_\mu$ as
$$
M_\mu f(x) = \sup_{r > 0} \frac{1}{\mu(B(x,r))} \int_{B(x,r)} |f| \, d\mu
\quad
\text{for all }
f \in L^1(\mu),
\, x \in \supp(\mu).
$$
When $\mu$ is doubling, it is well known (see e.g. Theorem~2.2 in \cite{HeinBook})
that 
$M_\mu$ is of weak type $(1,1)$.
\begin{lemma}
\label{l-glue}
Suppose $K$ is a kernel 
and $\mu$ is a regular measure on a metric space $M$
with $\supp(\mu) = A \cup B$ where $B$ is bounded and $d(A,B) > \diam(B)$.
If there is a constant $c>0$ such that $f \in L^1(\mu)$ and $\lambda > 0$ satisfy
$$
\mu(\{ x \in A : T_{\mu|_A,*} f(x) > \lambda \} ) \leq \frac{c}{\lambda} \Vert f \Vert_{L^1(\mu)}
$$
and
$$
\mu(\{ x \in B : T_{\mu|_B,*} f(x) > \lambda \} ) \leq \frac{c}{\lambda} \Vert f \Vert_{L^1(\mu)},
$$
then there is a constant $C>0$ depending only on $c$, $K$, and $\mu$ such that
$$
\mu(\{ x \in M : T_{\mu,*} f(x) > 2\lambda \} ) \leq \frac{C}{\lambda} \Vert f \Vert_{L^1(\mu)}.
$$

In particular, if $T_{\mu|_A,*}$ and are both $T_{\mu|_B,*}$ of weak type $(1,1)$,
then $T_{\mu,*}$ is of weak type $(1,1)$.
\end{lemma}
\begin{proof}
Write $\mu_1 := \mu|_A$ and $\mu_2 := \mu|_B$.
Since $T_{\mu,*}f \leq T_{\mu_1,*}f + T_{\mu_2,*}f$, we can write
\begin{align}
    \mu(\{ x \in M &: T_{\mu,*} f(x) > 2\lambda \} ) \nonumber \\
    &=
    \mu(\{ x \in A : T_{\mu,*} f(x) > 2\lambda \} ) 
    +
    \mu(\{ x \in B : T_{\mu,*} f(x) > 2\lambda \} ) \nonumber \\
    &\leq
     \mu(\{ x \in A : T_{\mu_1,*} f(x) > \lambda \} ) +  \mu(\{ x \in A : T_{\mu_2,*} f(x) > \lambda \} ) \nonumber \\
     & \qquad + 
     \mu(\{ x \in B : T_{\mu_1,*} f(x) > \lambda \} ) + \mu(\{ x \in B : T_{\mu_2,*} f(x) > \lambda \} ) \nonumber \\
     &:= m_{A,1} + m_{A,2} + m_{B,1} + m_{B,2} \label{e-0}.
\end{align}
By assumption, we have
\begin{equation}
    \label{e-1}
    m_{A,1} \leq \frac{c}{\lambda} \Vert f \Vert_{L^1(\mu)}
    \quad
    \text{ and }
    \quad
    m_{B,2} \leq \frac{c}{\lambda} \Vert f \Vert_{L^1(\mu)}.
\end{equation}
To bound $m_{B,1}$, note that, for any $x \in B$ and $\varepsilon > 0$, \eqref{e-growth-metric} gives
a constant $C_1>0$ such that
\begin{align*}
    |T_{\mu_1,\varepsilon} f(x)| &= \left|\int_{d(x,y) > \varepsilon} K(x,y) f(y) d \mu_1(y)\right|
    \leq C_1 \int_{A} \frac{|f(y)|}{d(x,y)} d \mu(y).
\end{align*}
Since $d(A,B) > \diam(B)$ and $x \in B$, we have the bound
\begin{align*}
\int_{A} \frac{|f(y)|}{d(x,y)} d \mu(y)
    \leq 
    \int_{d(x,y) > \diam(B)}  \frac{|f(y)|}{d(x,y)} d \mu(y) 
    \leq \frac{\Vert f \Vert_{L^1(\mu)}}{\diam(B)}.
\end{align*}
Therefore, $T_{\mu_1,*} f(x) \leq \diam(B)^{-1} \Vert f \Vert_{L^1(\mu)}$, and so, since $\mu$ is regular,
\begin{align}
\label{e-2}
    m_{B,1} 
    \leq 
    \frac{1}{\lambda} \int_B T_{\mu_1,*} f \, d\mu
    \leq
    \frac{C_1}{\lambda} \cdot \frac{\Vert f \Vert_{L^1(\mu)} \mu(B)}{\diam(B)}
    \leq
     \frac{C_1 \reg(\mu)}{\lambda}  \Vert f \Vert_{L^1(\mu)}.
\end{align}
To bound $m_{A,2}$, note that, for any $x \in A$, we have $d(x,B) \geq d(A,B) > \diam(B)$, and so, for any $\varepsilon > 0$,
\begin{align*}
    |T_{\mu_2,\varepsilon} f(x)| 
    \leq C_1 \int_{B} \frac{|f(y)|}{d(x,y)} d \mu(y)
    &\leq \frac{C_1}{d(x,B)} \int_{B(x,2d(x,B))} |f| \, d \mu \\
&\leq
\frac{2C_1 \reg(\mu)}{\mu(B(x,2d(x,B)))} \int_{B(x,2d(x,B))} |f| \, d \mu\\
&\leq 2C_1 \reg(\mu) M_\mu f(x).
\end{align*}
Since $\mu$ is a doubling measure, there is a constant $C_2>0$ depending only on $\mu$ and $C_1$ such that
\begin{align}
\label{e-3}
    m_{A,2} 
    \leq 
    \mu \left( \left\{ x \in A : M_{\mu} f(x) > \lambda(2C_1\reg(\mu))^{-1} \right\} \right)
    \leq
    \frac{C_2}{\lambda} \Vert f \Vert_{L^1(\mu)}.
\end{align}
Combining \eqref{e-1}, \eqref{e-2}, and \eqref{e-3} with \eqref{e-0} completes the proof.
\end{proof}

\section{$w^*$-differentiability on curves}
\label{s-weakstar}

Suppose $X$ is a Banach space which is dual to a separable Banach space $Y$ i.e. $X = Y^*$.
Recall the definition of the $w^*$-derivative from the introduction \cite{AmbKir,HeiMan}.

\begin{definition}
Suppose $\Omega \subset \mathbb{R}^n$ is open.
A map 
$f: \mathbb{R}^n \supset \Omega \to X$ is {\em $w^*$-differentiable} at $x \in \Omega$
if there is a linear map $D^*f_x:\mathbb{R}^n \to X$ such that
$$
\lim_{z \to x}
\left\langle \frac{f(z)-f(x)-D^*f_x(z-x)}{|z-x|},y \right\rangle = 0
\quad 
\text{for all } y \in Y.
$$
\end{definition}
For a curve $\gamma:[a,b] \to X$,
this translates to the following: $\gamma'(x) \in X$ is the $w^*$-{\em derivative} of $\gamma$ at $x \in (a,b)$ if
$$
\lim_{h \to 0}
\left\langle \frac{\gamma(x+h)-\gamma(x)-h \gamma'(x)}{h},y \right\rangle = 0
\quad
\text{for every } y \in Y.
$$
In other words, $\frac{d}{ds} \langle \gamma(s),y \rangle|_{s=x} = \langle \gamma'(x),y\rangle$ for any $y \in Y$.
We define $\gamma'(a)$ and $\gamma'(b)$ via one sided limits.

\subsection{A change of variables formula for the $w^*$-derivative}

To begin our exploration of the $w^*$-derivative, we prove a change of variables formula for Lipschitz curves in $X$.

We first record the powerful theorem of Ambrosio and Kirchheim.
\begin{theorem}[{\cite[Theorem~3.5]{AmbKir}}]
A Lipschitz map $\gamma:\mathbb{R}^n \to X$
is $w^*$-differentiable a.e.,
metrically differentiable a.e., and, for a.e. $x \in \mathbb{R}^n$,
$$
\Vert D^*f_x(v) \Vert 
=
md \, f_x(v)
\quad
\text{for all }v \in \mathbb{R}^n
$$
where $md\, f_x$ is metric differential of $f$ at $x$.
\end{theorem}
For curves in $X$, this theorem translates to the following.
For the remainder of this section,
$\mathcal{I}$ will be any (bounded or unbounded) interval in $\mathbb{R}$.
\begin{corollary}
\label{t-equalderiv}
A Lipschitz curve $\gamma:\mathcal{I} \to X$
is $w^*$-differentiable a.e. and 
$$
\Vert \gamma'(x) \Vert =
 |\dot{\gamma}|(x) 
:= \lim_{h \to 0} \frac{\Vert \gamma(x+h) - \gamma(x) \Vert}{|h|}
\quad
\text{for a.e. }x \in [a,b].
$$
\end{corollary}

We will now use this result to verify a change of variables formula for Lipschitz curves in $X$.
\begin{proposition}
\label{p-CoV}
If $\gamma:\mathcal{I} \to X$
is Lipschitz and
$\gamma$ is injective on $[a,b] \subseteq \mathcal{I}$,
then
$$
\int_{\gamma([a,b])} f \, d \mathcal{H}^1|_{\gamma([a,b])} = \int_a^b f(\gamma(t)) \Vert \gamma'(t)\Vert \, dt
\quad
\text{for all }
f \in L^1(\mathcal{H}^1|_{\gamma([a,b])}).
$$
\end{proposition}
\begin{proof}
According to \cite[Theorem 3.6]{HajSobMet}, for any $[c,d] \subseteq [a,b]$ we have
$$
\mathcal{H}^1(\gamma([c,d])) = \ell(\gamma|_{[c,d]})
=
\int_c^d |\dot{\gamma}|(t) \, dt
$$
where $\ell(\gamma|_{[c,d]})$ is the length of the curve 
$\gamma:[c,d] \to X$.
Corollary~\ref{t-equalderiv} therefore gives
$
\mathcal{H}^1(\gamma([c,d])) =
\int_c^d \Vert \gamma'(t)\Vert \, dt.
$
Using standard approximations of $L^1$
functions by characteristic functions $\chi|_{\gamma([c,d])}$
proves the proposition.
\end{proof}

\subsection{$w^*$-smooth curves}

\begin{definition}
\label{d-c1}
A curve $\gamma:\mathcal{I} \to X$
is $C_{w^*}^{1}$
if $\gamma'(x)$ exists for all $x \in \mathcal{I}$
and $\gamma':\mathcal{I} \to X$ is continuous.
\end{definition}

\begin{proposition}
\label{p-bilip}
If $\gamma:\mathcal{I} \to X$ 
is $C_{w^*}^1$ with 
$$
0 < m := \inf_{\mathcal{I}} \Vert \gamma' \Vert
\quad
\text{ and }
\quad
\sup_{\mathcal{I}} \Vert \gamma' \Vert =: M < \infty,
$$
then $\gamma$ is uniformly locally bi-Lipschitz.
That is, there is some $\delta > 0$ so that
$$
\tfrac{m}{2}|t_2-t_1| \leq \Vert \gamma(t_2) - \gamma(t_1) \Vert
\leq M |t_2-t_1|
\quad
\text{for all } t_1,t_2 \in \mathcal{I} \text{ with } |t_2-t_1| < \delta.
$$
\end{proposition}

\begin{proof}
For any $y \in Y$ and any $[t_1,t_2] \subseteq \mathcal{I}$, we have
\begin{align*}
    \langle \gamma(t_2) - \gamma(t_1) , y \rangle
    =
    \int_{t_1}^{t_2} \tfrac{d}{ds} \langle \gamma(s) , y \rangle \, ds
    =
    \int_{t_1}^{t_2} \langle \gamma'(s) , y \rangle \, ds.
\end{align*}
Taking the supremum over all $y \in Y$ with $\Vert y \Vert_Y = 1$ (first on the right then on the left) gives
$$
\Vert \gamma(t_2) - \gamma(t_1) \Vert
\leq 
\int_{t_1}^{t_2} \Vert \gamma'(s) \Vert \, ds
\leq  M |t_2 - t_1|.
$$

Now, by the continuity of $\gamma'$, there is some $\delta > 0$ so that $\Vert \gamma'(t_1) - \gamma'(t_2)\Vert<\tfrac{m}{2}$ whenever $|t_2 - t_1| < \delta$.
Fix $[t_1,t_2] \subseteq \mathcal{I}$ so that $|t_2-t_1|<\delta$ . Then for all $y \in Y$ we have
\begin{align*}
\left\langle \frac{\gamma(t_2)-\gamma(t_1)-(t_2-t_1) \gamma'(t_1)}{t_2-t_1},y \right\rangle
    &=
     \frac{1}{t_2-t_1}
     \int_{t_1}^{t_2} \langle \gamma'(s),y \rangle \, ds
    -
    \langle \gamma'(t_1),y \rangle\\
     &=
      \frac{1}{t_2-t_1}
    \int_{t_1}^{t_2} \langle \gamma'(s)-\gamma'(t_1),y \rangle \, ds.
\end{align*}
Again taking the supremum over all $y \in Y$ with $\Vert y \Vert_Y = 1$ gives
\begin{align*}
\left\Vert \frac{\gamma(t_2)-\gamma(t_1)-(t_2-t_1) \gamma'(t_1)}{t_2-t_1} \right \Vert
&\leq 
\frac{1}{t_2-t_1}
\int_{t_1}^{t_2} \Vert \gamma'(s) - \gamma'(t_1)\Vert \, ds
<
\tfrac{m}{2}.
\end{align*}
Therefore
$$
m|t_2-t_1| - \Vert \gamma(t_2) - \gamma(t_1)\Vert
\leq 
\Vert\gamma'(t_1)\Vert |t_2-t_1| - \Vert \gamma(t_2) - \gamma(t_1)\Vert \leq \tfrac{m}{2} |t_2-t_1| 
$$
for all $|t_1-t_2| < \delta$ in $\mathcal{I}$, and the result follows.
\end{proof}

The following corollary is immediate since the bilipschitz image of a regular set is itself regular (in the sense of \eqref{d-regular}), 
and $\Gamma = \gamma([a,b])$ is a finite union of bilipschitz images of intervals
when $C_{w^*}^1$ and
$\inf_{[a,b]} \Vert \gamma' \Vert > 0$.

\begin{corollary}
\label{c-regular}
If $\gamma: [a,b] \to X$ is
$C_{w^*}^1$ and
$\inf_{[a,b]} \Vert \gamma' \Vert > 0$,
then 
$\Gamma = \gamma([a,b])$ is regular.
\end{corollary}

\subsection{$C^{1,\alpha}_{w^*}$ curves}


\begin{definition}
\label{d-c1a}
For any $\alpha \in (0,1]$,
we say that $\gamma:\mathcal{I} \to X$ is $C_{w^*}^{1,\alpha}$
if it is $C_{w^*}^{1}$ and $\gamma'$ is $\alpha$-H\"{o}lder continuous
i.e. there is a constant $c > 0$ so that
$$
\Vert \gamma'(t_1) - \gamma'(t_2) \Vert
\leq c |t_1-t_2|^\alpha
\quad
\text{for all }
t_1,t_2 \in \mathcal{I}.
$$
\end{definition}

For a fixed $t_0 \in \mathcal{I}$,
call
$$
L_{t_0}(t) = \gamma(t_0)+(t-t_0)\gamma'(t_0)
$$
the {\em linear approximation of $\gamma$ at $t_0$}.

\begin{lemma}
\label{l-flat}
If $\gamma:\mathcal{I} \to X$ is $C_{w^*}^{1,\alpha}$,
then, for any $t_0 \in \mathcal{I}$,
\begin{equation}
\label{e-flat}
\Vert \gamma(t) - L_{t_0}(t) \Vert
\leq 
c|t-t_0|^{1+\alpha}
\end{equation}
where $c>0$ is the constant from the $\alpha$-H\"{o}lder continuity.
\end{lemma}

\begin{proof}
For any $y \in Y$, we have, as above
\begin{align*}
    \langle \gamma(t) - L_{t_0}(t),y \rangle
    =
    \langle \gamma(t) - \gamma(t_0) , y \rangle - (t-t_0) \langle \gamma'(t_0),y \rangle 
     &=
    \int_{t_0}^t \langle \gamma'(s)-\gamma'(t_0),y \rangle \, ds.
\end{align*}
Again, taking the supremum over all $y \in Y$ with $\Vert y \Vert_Y = 1$ gives
\begin{align*}
\Vert \gamma(t) - L_{t_0}(t) \Vert
\leq 
\int_{t_0}^t \Vert \gamma'(s)-\gamma'(t_0) \Vert \, ds
\leq 
c \int_{t_0}^t |s-t_0|^{\alpha} \, ds
\leq
c|t-t_0|^{1+\alpha}.
\end{align*}
\end{proof}

We will now define the class of curves to which Theorem~\ref{t-main} may be applied.

\begin{definition}
A set $\Gamma \subset X$
is a {\em $C_{w^*}^{1,\alpha}$ regular curve}
if 
$\Gamma = \gamma([a,b])$ for some
$C_{w^*}^{1,\alpha}$ curve $\gamma:[a,b] \to X$ with $\inf_{[a,b]} \Vert \gamma' \Vert > 0$.
\end{definition}

\begin{remark}
\label{r-unbound}
Note that, in Theorem~\ref{t-main}, we consider only curves defined on compact intervals (as in \cite{SIOCLZ}).
The reason for this restriction may be seen in the next section.
In particular, in the proof of Lemma~\ref{l-chooseG},
we rely on the fact that the curve $\Gamma$ has finite length.
Unfortunately, it is not clear to me how this lemma may be proven if $\Gamma$ is unbounded.
See the discussion following the proof of Lemma~\ref{l-chooseG} for more details.
\end{remark}

\section{Proof of Theorem~\ref{t-main}}
\label{sec-proof}

Theorem~\ref{t-main} will follow from the following ``good lambda'' argument 
a la David \cite{MR956767} and Tolsa \cite{tolsabook}
which allows us to reduce our attention in any ball to a large enough subset $G$.
The lemma stated here will follow from 
the more general result proven later 
(see Theorem~\ref{t:goodlambda}) which is proven for upper regular (but not necessarily doubling)
measures on doubling metric spaces.

\begin{lemma}
\label{t-goodl}
Suppose $X$ is a Banach space,
$\Gamma$ is a regular set in $X$,
and $K$ is a 1-dimensional CZ-kernel on $X$
which satisfies the UBL condition.
Assume that there exist constants 
$0< \theta <1$ and $c > 0$
such that,
for every ball $B = B(x,r)$ with $x \in \Gamma$ and $r>0$,
there is a compact set $G' \subset B \cap \Gamma$ 
such that $\mathcal{H}^1(G') \geq \theta \mathcal{H}^1(B \cap \Gamma)$ and 
\begin{equation}
    \label{e-weaktype}
\mathcal{H}^1\left( \{ x \in G' \, : \, T_{\mathcal{H}^1|_\Gamma,*}f(x) > \lambda\} \right)
\leq \frac{c}{\lambda}\Vert f \Vert_{L^1(\mathcal{H}^1|_\Gamma)}
\quad
\forall f \in L^1 (\mathcal{H}^1|_\Gamma), \, \lambda>0.
\end{equation}
Then there is a constant $A_p'=A_p'(c, \Gamma, K, \theta)$ for each $p \in (1,\infty)$ such that 
$$
\Vert T_{\mathcal{H}^1|_\Gamma,*} f \Vert_{L^{p}(\mathcal{H}^1|_\Gamma)} \leq A_p' \Vert f \Vert_{L^p(\mathcal{H}^1|_\Gamma)}.
$$
\end{lemma}
\begin{proof}
Our goal will be to
apply Theorem~\ref{t:goodlambda} to the doubling metric space $M = \Gamma$.
If $\diam(\Gamma)= \infty$, then the result follows immediately.

Now suppose that $\diam(\Gamma) < \infty$.
In this case, 
we may not directly apply Theorem~\ref{t:goodlambda} since 
$\supp(\mathcal{H}^1|_\Gamma) = \Gamma$ is bounded.
To remedy this, we will append rays to $\Gamma$ (as in \cite{SIOCLZ}) and apply Theorem~\ref{t:goodlambda} to this new, unbounded support.
We can assume without loss of generality that $0 \in \Gamma$. 
Indeed, the left invariance of the metric in a normed vector space implies that $\mathcal{H}^1|_E = \mathcal{H}^1|_{x + E}$ for any $x \in X$ and $E \subset X$.

Since $\Gamma$ is regular, it follows that, for any ball $B$ centered on $\Gamma$ 
with radius $r = \diam(\Gamma) < \infty$,
$$
\mathcal{H}^1(\Gamma) = \mathcal{H}^1(\Gamma \cap B) \leq \reg(\Gamma) \diam(\Gamma) < \infty.
$$
Thus, by rescaling the norm, we may assume without loss of generality that $\mathcal{H}^1(\Gamma)=1$.

Consider the ray $\ell = \{ sv_0 : s \in [3,\infty) \}$ where $v_0 \in X$ satisfies $\Vert v_0 \Vert = 1$.
Set $\hat{\Gamma} = \Gamma \cup \ell$ and $\nu = \mathcal{H}^1|_{\hat{\Gamma}}$.
Since $\Gamma \subset B(0,2)$ and $\ell \cap B(0,2) = \emptyset$,
the measure $\nu$ is regular with $\reg(\nu)$ depending only on $\Gamma$.
Moreover, $\supp(\nu)$ is unbounded.

Fix $x \in \hat{\Gamma}$ and choose $r > 0$. Set $B = B(x,r) \cap \hat{\Gamma}$.
In order to apply Theorem~\ref{t:goodlambda},
it remains to prove that $B$
has a compact subset $G \subset B$
with
$\nu(G) \geq \theta \nu(B)$ 
satisfying
\begin{equation}
    \label{e:goodlamb0}
\nu\left( \{ x \in G \, : \, T_{\nu,*}f(x) > \lambda\} \right)
\leq \frac{c}{\lambda}\Vert f \Vert_{L^1(\nu)}
\quad
\text{for all } f \in L^1(\nu), \, \lambda > 0
\end{equation}
for a constant $c>0$ depending only on $K$, $\nu$, and $X$.

\noindent \textbf{Case 1:} $x \in \ell$.

Choose $G = B \cap \ell$.
If $r < 1$, then $B$ does not intersect the original set $\Gamma$
since $0 \in \Gamma$, $\diam(\Gamma) \leq \mathcal{H}^1(\Gamma) = 1$, and $d(\ell,0) = 3$.
Thus $\nu(B) = \nu(B \cap \ell) = \nu(G)$.
If $r \geq 1$, then 
$$
\nu(B) \leq \nu(\Gamma) + \nu(B \cap \ell) 
= 1 + \nu(G)
\leq (\reg(\nu) + 1) \nu(G).
$$
According to the UBL assumption on $K$,
there is a constant $C>0$ independent of $B$ such that
$$
\Vert T_{\mathcal{H}^1|_{\ell}} \Vert_{L^2(\mathcal{H}^1|_{\ell}) \to L^2(\mathcal{H}^1|_{\ell})} \leq C < \infty.
$$
Since $G \subset \ell$ and $\mathcal{H}^1|_G =\chi_G \, d \mathcal{H}^1|_{\ell}$, 
it follows that $\Vert T_{\mathcal{H}^1|_G} \Vert_{L^2(\mathcal{H}^1|_G) \to L^2(\mathcal{H}^1|_G)} \leq C$ as well.
Indeed, for any $f \in L^2(\mathcal{H}^1|_G)$ and $\varepsilon>0$,
\begin{align*}
\Vert T_{\mathcal{H}^1|_G,\varepsilon} f \Vert_{L^2(\mathcal{H}^1|_G)}
=
\Vert T_{\mathcal{H}^1|_\ell,\varepsilon} (f \chi_G) \Vert_{L^2(\mathcal{H}^1|_G)}
&\leq
\Vert T_{\mathcal{H}^1|_\ell,\varepsilon} (f \chi_G) \Vert_{L^2(\mathcal{H}^1|_\ell)}\\
&\leq
C \Vert f \chi_G \Vert_{L^2(\mathcal{H}^1|_\ell)}
=
C \Vert f  \Vert_{L^2(\mathcal{H}^1|_G)}.
\end{align*}
Thus
Lemma~\ref{t-L2toWeak}
applied to the metric space $G$ with measure $\mu = \nu|_G = \mathcal{H}^1|_G$
implies \eqref{e:goodlamb0}.

\noindent \textbf{Case 2:} $x \in \Gamma$



By our assumptions, 
there is a compact set $G' \subset B \cap \Gamma$
such that $\mathcal{H}^1(G') \geq \theta \mathcal{H}^1(B \cap \Gamma)$ and $T_{\mathcal{H}^1|_\Gamma,*}$ 
satisfies \eqref{e-weaktype} on $G'$ with constant independent of $B$.
Set $G = G' \cup (B \cap \ell)$.
We have
$$
\nu(G) = \nu(G') + \nu(B \cap \ell) \geq \theta\nu(B\cap \Gamma) + \theta\nu( B \cap \ell) = \theta \nu(B)
$$
since $\theta < 1$.
Moreover, the UBL assumption once again implies that
$$
\Vert T_{\mathcal{H}^1|_{\ell}} \Vert_{L^2(\mathcal{H}^1|_{\ell}) \to L^2(\mathcal{H}^1|_{\ell})} \leq C < \infty
$$
for a constant $C>0$ independent of $B$.
Lemma~\ref{t-L2toWeak} again implies that the operator $T_{\mathcal{H}^1|_{\ell},*}$ 
is of weak type $(1,1)$ 
on $\ell$.
Thus Lemma~\ref{l-glue} 
applied to the measure $\mu=\nu$ and the sets $A = \ell$ and $B= \Gamma$
ensures that $T_{\nu,*}$ 
satisfies \eqref{e:goodlamb0}.

Therefore, 
for any ball $B$ centered on $\hat{\Gamma}$,
we may use the arguments above to define a compact set $G \subset B \cap \hat{\Gamma}$ 
such that $\nu(G) \geq \theta \nu(B)$ and 
\begin{equation*}
\nu\left( \{ x \in G \, : \, T_{\nu,*}f(x) > \lambda\} \right)
\leq \frac{c}{\lambda}\Vert f \Vert_{L^1(\nu)}
\quad
\text{for all } f \in L^1(\nu), \, \lambda>0.
\end{equation*}
According to Theorem~\ref{t:goodlambda},
there is a constant $A_p=A_p(c, \Gamma, K, \theta)$ for each $p \in (1,\infty)$ such that 
$\Vert T_{\nu,*} f \Vert_{L^{p}(\nu)} \leq A_p \Vert f \Vert_{L^p(\nu)}$.
Since $\mathcal{H}^1|_\Gamma = \chi_{\Gamma} \, d \mathcal{H}^1|_{\hat{\Gamma}} = \chi_{\Gamma} \, d\nu$, 
we have proven the lemma.
\end{proof}

We will now present the proof of Theorem~\ref{t-main}.
Most of the work has already been done in Section~\ref{s-weakstar}, Lemma~\ref{t-goodl}, and \cite{SIOCLZ}.
We will simply see how the pieces fit together.

\begin{proof}[Proof of Theorem~\ref{t-main}]
Suppose without loss of generality that
$\Gamma = \gamma([0,1])$ where $\gamma:[0,1] \to X$ is $C_{w^*}^{1,\alpha}$ and $\Vert \gamma'\Vert \neq 0$ on $[0,1]$.
According to Corollary~\ref{c-regular}, 
$\Gamma$ is regular.
Assume again without loss of generality that $\mathcal{H}^1(\Gamma) \leq \ell(\gamma) = 1$.
Suppose $E \subset \Gamma$ is a regular set.
As before, since $\mathcal{H}^1|_E = \chi_E \, d \mathcal{H}^1|_\Gamma$, it suffices to prove that $T_{\mathcal{H}^1|_\Gamma}$ is bounded on $L^p(\mathcal{H}^1|_\Gamma)$ for $1<p<\infty$.

Write $\nu = \mathcal{H}^1|_\Gamma$.
We will now describe how to choose $G$ inside balls centered on $\Gamma$ 
so that Lemma~\ref{t-goodl} may be applied.

\begin{lemma}
\label{l-chooseG}
There is a constant $0<\theta<1$
such that,
for every ball $B$
centered on $\Gamma$,
there is an interval $[a,b] \subset [0,1]$
such that 
$\gamma$ is bilipschitz on $[a,b]$,
$G := \gamma([a,b]) \subset B \cap \Gamma$, 
and $\nu(G) \geq \theta \nu(B \cap \Gamma)$.
Moreover,
\begin{equation}
    \label{e-chooseG}
\diam\left([a,b] \cap \gamma^{-1}(B(y,s) \cap G)\right) \leq \tfrac{2s}{\inf_{[0,1]} \Vert \gamma' \Vert} \text{ for all } y\in G, \, s>0.
\end{equation}
\end{lemma}
\begin{proof}
Without loss of generality, we may assume $r < 2$ since $\diam(\Gamma) \leq \ell(\gamma) = 1$.
According to Proposition~\ref{p-bilip},
there is some 
$\delta \in (0,1)$ 
independent of the choice of $B$ such that 
\begin{equation}
    \label{e-bilip}
    \tfrac{m}{2}|t_1 - t_2| \leq \Vert \gamma(t_1) - \gamma(t_2) \Vert \leq M|t_1 - t_2|
    \quad \text{for all } |t_1-t_2|<\delta
\end{equation}
where $m = \min \{ 1, \inf_{[0,1]} \Vert \gamma' \Vert \}$
and
$M = \max \{ 1, \sup_{[0,1]} \Vert \gamma' \Vert \}$.
Choose $a,b \in [0,1]$
so that $\gamma(a) = x$
and $|b-a| 
= \frac{\delta \nu (B \cap \Gamma)}{2M  \text{reg}(\nu)}$. 
Note that $|b-a| < \delta$
since $\nu(B \cap \Gamma) < 2 \text{reg}(\nu)$. 
Without loss of generality, assume $a<b$.
Set $G = \gamma([a,b])$.
Note first that $G \subset B$ since 
$$
\Vert \gamma(t) - x \Vert
=
\Vert \gamma(t) - \gamma(a) \Vert \leq M|b-a| < \frac{\nu (B \cap \Gamma)}{\text{reg}(\nu)} \leq r
\quad
\text{ for all } t \in [a,b]
$$
where $r>0$ is the radius of $B$.
Equation \eqref{e-bilip} implies
\eqref{e-chooseG} and also gives
$$
\nu(G) \geq \Vert \gamma(b) - \gamma(a) \Vert
\geq \tfrac{m}{2} |b-a| = \frac{m\delta \nu (B \cap \Gamma)}{4M  \text{reg}(\nu)}.
$$
Setting $\theta = \frac{m \delta}{4M  \text{reg}(\nu)} < 1$
completes the proof of the lemma.
\end{proof}

\begin{remark}
It is at this point that our proof of Theorem~\ref{t-main}
might not carry over to the case in which the domain of $\gamma$ is unbounded.
(See Remark~\ref{r-unbound}.)
One must be able to choose a large enough interval in the domain which maps inside $B$.
In particular, $\gamma$ should be injective on this interval so that \eqref{e-chooseG} can hold.
However, we only know that $\gamma$ is {\em locally} bi-Lipschitz.
This leads to the following question whose answer would determine whether or not Theorem~\ref{t-main} can be proven on unbounded domains.
\end{remark}
\begin{question}
Does Lemma~\ref{l-chooseG} hold when $\gamma$ is defined on an unbounded interval?
\end{question}

We now return to the proof of Theorem~\ref{t-main}.
Fix a ball $B$ centered on $\Gamma$
and
choose $G \subset B \cap \Gamma$ as in Lemma~\ref{l-chooseG}.
In order to apply Lemma~\ref{t-goodl} and conclude the proof, we must verify \eqref{e-weaktype} on $G$.
That is, we must show that $T_{\nu,*}$ is uniformly of weak type $(1,1)$ on $G$. 
According to Lemma~\ref{t-L2toWeak}
applied to the separable metric space $G$,
it suffices to prove that 
$
\Vert T_{\nu} \Vert_{L^2(\nu|_{G}) \to L^2(\nu|_{G})}
$
is bounded by a constant depending only on $\gamma$, $K$, and $X$.
In other words, it remains to prove the following:
\begin{proposition}
\label{p-L2goal}
There is a constant $C > 0$ depending only on $X$, $K$, and $\gamma$ such that
\begin{equation}
    \label{e-L2goal}
    \Vert T_{\mathcal{H}^1|_{G},\varepsilon}f\Vert_{L^2(\mathcal{H}^1|_{G})} \leq C \Vert f\Vert_{L^2(\mathcal{H}^1|_{G})}
    \quad
    \text{for all } \varepsilon > 0.
\end{equation}
\end{proposition}

The proof of \eqref{e-L2goal}
is nearly identical to that of (4.13) in \cite{SIOCLZ}.
We refer the reader to the proof contained therein with the following substitutions: 
Proposition~2.16,
Proposition~3.1,
Proposition~3.2,
and (4.10)
in \cite{SIOCLZ} are replaced here by
Proposition~\ref{p-annular},
Proposition~\ref{p-CoV},
Lemma~\ref{l-flat},
and 
(4.2)
respectively.
This completes the proof of Theorem~\ref{t-main}.
\end{proof}

\section{Good lambda for upper regular measures on doubling metric spaces}
\label{s-bigpiece}

David originally presented his ``good lambda'' argument for regular measures in $\mathbb{R}^n$ \cite{MR956767}.
A version of David's result recently appeared in \cite{FasOrpReg} for regular
measures in any proper metric space.
Note that every regular measure is necessarily doubling.
In \cite[Theorem 2.22]{tolsabook}, Tolsa proved a version of David's result for upper regular 
Radon measures in $\mathbb{R}^n$ which are not necessarily doubling.

Theorem~\ref{t:goodlambda} below generalizes Tolsa's result to the setting of upper regular Radon measures on doubling metric spaces.
As in \cite{tolsabook}, this result does not require that the measure itself is doubling.

We say that a metric space $M$ is {\em doubling} if there is a positive integer $C_D$ 
such that any ball $B(x,r) \subset M$ can be covered by $C_D$ balls of radius $r/2$.
A Borel measure $\nu$ on $M$ is upper $n$-regular if
there is a constant $C_\nu>0$ such that
$$
\nu (B(x,r)) \leq C_\nu r^n 
\quad
\text{for all } 
x \in M, \, r>0.
$$
For a constant $b \geq 1$, we say that a ball $B \subset M$ is {\em $b$-doubling} if $\nu(2B) \leq \beta \nu(B)$.
\begin{theorem}
\label{t:goodlambda}
Suppose $(M,d)$ is a doubling metric space, 
$K$ is an $n$-dimensional CZ-kernel on $M$,
and $\nu$ is an upper $n$-regular Radon measure on $M$ with 
unbounded support.
Assume $c > 0$ and $0< \theta <1$ are constants 
such that,
for some $b > 0$,
every $b$-doubling ball $B$ centered on $\supp (\nu)$ 
has a compact subset $G \subset B \cap \supp (\nu)$ 
with $\nu(G) \geq \theta \nu(B)$ and 
\begin{equation}
    \label{e:goodlamb}
\nu\left( \{ x \in G \, : \, T_{\nu,*}f(x) > \lambda\} \right)
\leq \frac{c}{\lambda}\Vert f \Vert_{L^1(\nu)}
\quad
\text{for all } f \in L^1(\nu), \, \lambda > 0.
\end{equation}
Then for each $p \in (1,\infty)$,
there is a constant $A_p$ depending only on $c$, $C_\nu$, $\theta$, and the constant from \eqref{e-growth-metric} and \eqref{e-growth-holder} such that 
$$
\Vert T_{\nu,*} f \Vert_{L^{p}(\nu)} \leq A_p \Vert f \Vert_{L^p(\nu)}.
$$
\end{theorem}

We will prove Theorem~\ref{t:goodlambda} via the lemmas below.
Some arguments are standard
and follow the proof of Theorem~2.22 from \cite{tolsabook}, 
but details are included for the benefit of the reader.
Most importantly, we will establish the following ``good $\lambda$'' result.
\begin{lemma}
\label{l:goodlambda}
Suppose $M$, $K$, $n$, $c$, $\theta$, and $\nu$ satisfy the assumptions of Theorem~\ref{t:goodlambda}.
For any $\varepsilon > 0$, there is a $\delta > 0$ such that
\begin{align}
\label{e:lambda}
\nu(\{x \in \supp(\nu) : T_{\nu,*}f(x) > (1+\varepsilon)\lambda, & \, M_\nu f(x) \leq \delta \lambda\}) \nonumber \\
&\leq 
\left(1-\tfrac{\theta}{4}\right) \nu( \{ x \in \supp(\nu) : T_{\nu,*} f(x) > \lambda \} )
\end{align}
for any $\lambda > 0$ and any $f \in L^1(\nu)$ with compact support. 
\end{lemma}

Before proving this lemma, we will establish a few preliminary lemmas.
The first is standard.
\begin{lemma}
\label{l-tech}
If $\nu$ is an upper $n$-regular measure on a metric space $M$, 
and if $f \in L^p(\nu)$ for some $1<p<\infty$, then
    \begin{equation}
        \label{e-tech}
    \int_{M \setminus B(x,R)} \frac{|f(y)|}{d(x,y)^{n + \beta}} \, d\nu(y)
    \leq
    \tfrac{2^{n+\beta}}{C_\nu (2^\beta - 1)} R^{-\beta} M_\nu f(x)
    \quad
    \forall x \in \supp(\nu), \, R>0.
    \end{equation}
\end{lemma}

\begin{proof}
\begin{align*}
   \int_{M \setminus B(x,R)} \frac{|f(y)|}{d(x,y)^{n + \beta}} \, d \nu(y)
    &=
    \sum_{k=0}^{\infty} \int_{B(x,2^{k+1}R) \setminus B(x,2^kR)} \frac{|f(y)|}{d(x,y)^{n + \beta}} \, d\nu(y) \\
    &\leq
    \sum_{k=0}^{\infty} \frac{1}{(2^kR)^{n + \beta}} \int_{B(x,2^{k+1}R)}|f(y)| \, d\nu(y) \\
    &\leq C_\nu
    \sum_{k=0}^{\infty} \frac{(2^{k+1}R)^n}{(2^kR)^{n + \beta}} M_\nu f(x) 
    \leq
    \tfrac{C_\nu 2^{n+\beta}}{2^\beta - 1} R^{-\beta} M_\nu f(x).
\end{align*}
\end{proof}

We will also need the following Whitney-type decomposition of an open set in a doubling metric space $M$.

\begin{lemma}[Christ-Whitney decomposition]
\label{l-whitney}
Assume $M$ is a doubling metric space.
Suppose $\Omega \subset M$ is open and $\Omega \neq M$.
Then there is a countable collection $\mathcal{U} = \{ U_i \}_{i \in \mathbb{N}}$ of pairwise disjoint, Borel subsets of $\Omega$
such that
\begin{enumerate}
    \item $\bigcup \mathcal{U} = \Omega$,
    \item for any $i \in \mathbb{N}$, 
    there is some $x_i \in U_i$ and some integer $k_i$ such that
    $$
    B\left(x_i,\tfrac{1}{16} \cdot 8^{-k_i}\right)
    \subset
    U_i
    \subset
    B\left(x_i,2\cdot 8^{-k_i}\right)
    $$
    and 
    $$
    32 \cdot 8^{-k_i} \leq d(U_i,M \setminus \Omega) \leq 320 \cdot 8^{-k_i},
    $$
    \item
    there is a constant $C_D \geq 1$ depending only on the doubling constant of $M$ such that,
    for any $i \in \mathbb{N}$, there are at most $C_D$ indices $j \in \mathbb{N}$ such that 
    \begin{equation}
    \label{e-ballhit}
    B\left(x_i,16 \cdot 8^{-k_i}\right) \cap B\left(x_j,16 \cdot 8^{-k_j}\right) \neq \emptyset.
    \end{equation}
\end{enumerate}
\end{lemma}

\begin{proof}
According to Theorem~2.1 in \cite{KaRaSu}, 
there is a countable collection $\mathcal{Q} = \{ Q_{k,i}  : k \in \mathbb{Z}, \, i \in N_k \subset \mathbb{N} \}$ of Borel subsets of $\Omega$ such that
\begin{enumerate}[label=(\alph*)]
    \item $\Omega = \bigcup_{i \in N_k} Q_{k,i}$ for every $k \in \mathbb{Z}$,
    \item $Q_{k,i} \cap Q_{m,j} = \emptyset$
    or
    $Q_{k,i} \subset Q_{m,j}$
    when $k,m \in \mathbb{Z}$, $k \geq m$, $i \in N_k$, and $j \in N_m$,
    \item
    for any $k \in \mathbb{Z}$ and $i \in N_k$, 
    there is some $x_{k,i} \in Q_{k,i}$ such that
    $$
    B\left(x_{k,i},\tfrac{1}{16} \cdot 8^{-k}\right) \subset
    B\left(x_{k,i},\tfrac{5}{14} \cdot 8^{-k}\right)
    \subset
    Q_{k,i}
    \subset
    \overline{B}\left(x_{k,i},\tfrac87 \cdot 8^{-k}\right)
    \subset B\left(x_{k,i},2 \cdot 8^{-k}\right)
    $$
\end{enumerate}
where $\overline{B}(x,r) = \{ y \in M : d(x,y) \leq r \}$.
(Indeed, choose $r = \tfrac18$ in \cite{KaRaSu}.)
Now, write $\Omega_k := \{ x \in \Omega : 40 \cdot 8^{-k} < d(x,M \setminus \Omega) \leq 40 \cdot 8^{-(k-1)} \}$ so that $\Omega = \bigcup_{k=-\infty}^{\infty} \Omega_k$.
Define the collection 
$$
\mathcal{U}_0 := \bigcup_{k=-\infty}^\infty \{ Q_{k,i} : i \in N_k, \, Q_{k,i} \cap \Omega_k \neq \emptyset \}.
$$
For each $Q_{k,i} \in \mathcal{U}_0$,
we have
$d(Q_{k,i}, M \setminus \Omega) \leq 40 \cdot 8^{-(k-1)} = 320 \cdot 8^{-k}$.
Also,
$$
d(Q_{k,i}, M \setminus \Omega) 
\geq 40 \cdot 8^{-k} -  \diam (Q_{k,i})
\geq (40- \tfrac{16}{7})  8^{-k} > 32 \cdot 8^{-k}.
$$
Therefore, the collection $\mathcal{U}_0$ 
satisfies condition {\em (2)}.

Note that, if $Q \in \mathcal{U}_0$,
then there is a smallest integer $k$ and a unique $i \in N_k$ such that $Q \subset Q_{k,i}$.
We may thus define $\mathcal{U}$ to be the subset of $\mathcal{U}_0$ which is maximal with respect to set inclusion.
The resulting collection of sets is pairwise disjoint and satisfies conditions {\em (1)} and {\em (2)}

We will now verify that $\mathcal{U}$ satisfies {\em (3)}.
Write $\Omega^c := M \setminus \Omega$.
Assume $i,j \in \mathbb{N}$ satisfy \eqref{e-ballhit}.
Then 
\begin{align}
\label{e-dist1}
    d(x_i,x_j) \leq 16 \cdot 8^{-k_i} + 16 \cdot 8^{-k_j}
    \leq 
    \tfrac12 d(U_i,\Omega^c) + \tfrac12 d(U_j,\Omega^c).
\end{align}
Thus
\begin{align*}
    d(U_j,\Omega^c) \leq d(U_i,\Omega^c) + d(U_i,U_j)
    \leq 
    d(U_i,\Omega^c) + \tfrac12 d(U_i,\Omega^c) + \tfrac12 d(U_j,\Omega^c).
\end{align*}
Therefore
\begin{align*}
    d(U_j,\Omega^c) \leq 3 d(U_i,\Omega^c).
\end{align*}
Similarly,
\begin{align*}
    d(U_i,\Omega^c) \leq 3 d(U_j,\Omega^c).
\end{align*}
These two inequalities together with \eqref{e-dist1} give
\begin{align}
\label{e-whitbound2}
    d(x_i,x_j) \leq 2 d(U_i,\Omega^c)
    \leq 
    640 \cdot 8^{-k_i}
\end{align}

Now fix $i \in \mathbb{N}$.
If $j,\ell \in \mathbb{N}$ are such that
\begin{equation}
        \label{e-intersect}
        \begin{array}{c}
         B\left(x_i,16 \cdot 8^{-k_i}\right) \cap B\left(x_j,16 \cdot 8^{-k_j}\right) \neq \emptyset      \\
         \text{and}\\
         B\left(x_i,16 \cdot 8^{-k_i}\right) \cap B\left(x_\ell,16 \cdot 8^{-k_\ell}\right) \neq \emptyset     
        \end{array}
\end{equation}
    then, as above,
    \begin{align*}
    d(U_i,\Omega^c) \leq 3 d(U_j,\Omega^c)
        \quad
    \text{and}
\quad
d(U_i,\Omega^c) \leq 3 d(U_\ell,\Omega^c).
\end{align*}
Therefore, since $U_j$ and $U_\ell$ are disjoint, we have
\begin{align}
d(x_j,x_\ell) 
\geq 
\tfrac{1}{16} \min\{ 8^{-k_j},  8^{-k_\ell} \}
&\geq
\tfrac{1}{5120} \min\{ d(U_j,\Omega^c),d(U_\ell,\Omega^c) \} \nonumber \\
&\geq 
\tfrac{1}{15360} d(U_i,\Omega^c). \label{e-whitbound}
\end{align}
Consider now the ball $B := B(x_i, 1024 \cdot 8^{-k_i})$.
Since $M$ is doubling,
we can cover $B$ by $C_D$ balls $\{ V_\alpha \}$ of radius 
$\tfrac{1}{1024} \cdot 8^{-k_i}$
where $C_D>0$ is a positive integer depending only on the doubling constant of $M$.
According to
\eqref{e-whitbound2},
if $j \in \mathbb{N}$ satisfies \eqref{e-ballhit},
then 
$d(x_i,x_j) \leq
    640 \cdot 8^{-k_i}$,
so $x_j \in B$.
Moreover, for any $j,\ell \in \mathbb{N}$ which satisfy \eqref{e-intersect},
we must have from \eqref{e-whitbound} that
\begin{align*}
    d(x_j,x_\ell) \geq 
    \tfrac{1}{15360} d(U_i,\Omega^c)
    \geq
    \tfrac{1}{480} \cdot 8^{-k_i}.
\end{align*}
In particular, this means that $x_j$ and $x_\ell$ cannot lie in the same ball $V_\alpha$,
so there are at most $C_D$ indices $j \in \mathbb{N}$ which satisfy \eqref{e-ballhit}. 
This proves the lemma.
\end{proof}

We will now prove the ``good $\lambda$'' estimate Lemma~\ref{l:goodlambda}.

\begin{proof}[Proof of Lemma \ref{l:goodlambda}]
Assume $c > 0$ and $0< \theta <1$ are constants
which satisfy the hypotheses of the lemma,
and set $b = 2C_D$
where $C_D$ is the constant from Lemma~\ref{l-whitney}.
Fix $f \in L^1(\nu)$ with compact support.




Write $\mathcal{S} := \supp(\nu)$.
Consider the set
$$
\Omega_\lambda = \{ x \in \mathcal{S} \, : \, T_{\nu,*}f(x) > \lambda\}.
$$
Since $\diam(\mathcal{S}) = \infty$, the growth condition \eqref{e-growth-metric}
ensures that
$\Omega_\lambda \neq \mathcal{S}$.
Note also that
$\Omega_\lambda$ is open in $\mathcal{S}$.
Indeed, 
for any $\varepsilon > 0$,
the continuity of 
$T_{\nu,\varepsilon}f$ follows from the continuity of $K$,
and so $T_{\nu,*}f$ is lower semicontinuous.

Choose a
countable collection $\mathcal{U} = \{ U_i \}_{i \in \mathbb{N}}$ of pairwise disjoint, Borel subsets of $\Omega_\lambda \subsetneq \mathcal{S}$
as in
Lemma~\ref{l-whitney}.
Define $\mathcal{I}_1$ to be the collection of those indices $i \in \mathbb{N}$ such that 
\begin{equation}
    \label{e-dubstep}
\nu\left(B(x_i,16 \cdot 8^{-k_i}) \right)
\leq 
b \nu\left(B(x_i,2 \cdot 8^{-k_i}) \right).
\end{equation}
(Here and for the rest of this proof, $B(x,r)$ will refer to an open ball 
in the metric space $(\mathcal{S},d)$.)
Note in particular that any ball 
$B(x_i,2 \cdot 8^{-k_i})$
satisfying \eqref{e-dubstep} is $b$-doubling.
Write $\mathcal{I}_2 = \mathbb{N} \setminus \mathcal{I}_1$.

\noindent \textbf{Claim.}
\begin{equation}
    \label{e-morethanhalf}
    \sum_{i \in \mathcal{I}_1} \nu \left(  U_i \right)
\geq
\tfrac{1}{2} \nu(\Omega_\lambda)
\end{equation}
Indeed, for each $i \in \mathcal{I}_2$ we have
$$
b \nu\left(B(x_i,2 \cdot 8^{-k_i}) \right) 
< 
\nu\left(B(x_i,16 \cdot 8^{-k_i}) \right).
$$
Since any point in $\Omega_\lambda$
lies in at most $C_D$ balls of the form $B(x_i,16 \cdot 8^{-k_i})$
with $i \in \mathbb{N}$,
it follows that
$
\sum_{i \in \mathcal{I}_2}
\nu\left(U_i \right) 
$
is bounded by
$$
\sum_{i \in \mathcal{I}_2}
\nu\left(B(x_i,2 \cdot 8^{-k_i}) \right) 
\leq
\tfrac{1}{b}
\sum_{i \in \mathbb{N}}
\nu\left(B(x_i,16 \cdot 8^{-k_i}) \right)
\leq 
\tfrac{C_D}{b} \nu\left( \Omega_\lambda \right)
=
\tfrac12 \nu\left( \Omega_\lambda \right)
$$
since $b = 2C_D$.
This proves the claim.

\begin{lemma}
For any $\varepsilon > 0$, there is a $\delta > 0$ 
such that, 
if $i \in \mathbb{N}$ and $x \in U_i$ satisfy
$$
T_{\nu,*}f(x) > (1 + \varepsilon) \lambda
\quad
\text{ and }
\quad
M_{\nu}f(x) \leq \delta \lambda,
$$
then
$$
T_{\nu,*}(\chi_{B(x_i,4 \cdot 8^{-k})}f)(x) > \varepsilon \lambda/2.
$$
\end{lemma} 
\begin{proof}
Let $\varepsilon>0$.
Fix an index $i \in \mathbb{N}$, 
and write $U := U_i$, $x_0 := x_i$, and $k:=k_i$
so that
$B_1 := B\left(x_0,\tfrac{1}{16} \cdot 8^{-k_i}\right)$ and $B_2 := B\left(x_0,2 \cdot 8^{-k}\right)$ satisfy $B_1 \subset U \subset B_2$.
Suppose $x \in U$ satisfies $T_{\nu,*}f(x) > (1 + \varepsilon) \lambda$.

Since $d(U,\mathcal{S} \setminus \Omega_\lambda) \leq 320 \cdot 8^{-k}$ and $U \subset B_2$ with $r(B_2) = 2 \cdot 8^{-k}$,
there must be some $z \in \mathcal{S} \setminus \Omega_\lambda$
such that $d(z,x_0) < 322 \cdot 8^{-k}$.
Write $B := B(z,324 \cdot 8^{-k})$.
Notice that $x \in B$, and
\begin{align*}
    T_{\nu,*} (\chi_{\mathcal{S} \setminus 2B_2} f)(x) &\leq
    T_{\nu,*} (\chi_{2B \setminus 2B_2} f)(x)
    +
    T_{\nu,*} (\chi_{\mathcal{S} \setminus 2B} f)(x) \\
    &\leq
    T_{\nu,*} (\chi_{2B \setminus 2B_2} f)(x) +
    T_{\nu,*} (\chi_{\mathcal{S} \setminus 2B} f)(z) \\
    &\qquad + 
    |T_{\nu,*} (\chi_{\mathcal{S} \setminus 2B} f)(x)
    -
    T_{\nu,*} (\chi_{\mathcal{S} \setminus 2B} f)(z)|.
\end{align*}

We will bound these last three terms.
First,
the definitions of $T_{\nu,*}$ and $M_\nu$ together with \eqref{e-growth-metric} give
\begin{align*}
T_{\nu,*}(\chi_{2B \setminus 2B_2}f)(x)
\leq
C_K \int_{2B \setminus 2B_2} \frac{|f(y)|}{d(x,y)^n} \, d\nu(y)
&\leq
\frac{C_K}{(4 \cdot 8^{-k})^n}
\int_{B(x,648 \cdot 8^{-k})} |f|\\
&\leq
(162^n C_KC_\nu) M_\nu f(x).
\end{align*}

Moreover, since $2B$ is centered at $z$ and $z \notin \Omega_\lambda$,
$$
T_{\nu,*}(\chi_{\mathcal{S} \setminus 2B}f)(z)
\leq 
T_{\nu,*}f(z)
\leq
\lambda.
$$

Thirdly,
since $x,z \in B$ and
$
d(x,z) \leq 324 \cdot 8^{-k} \leq d(z,y)/2
$
when $y \in \mathcal{S} \setminus 2B$,
we have
from \eqref{e-growth-holder} that
\begin{align*}
    |
T_{\nu,*}(\chi_{\mathcal{S} \setminus 2B}f)(z)
&-
T_{\nu,*}(\chi_{\mathcal{S} \setminus 2B}f)(x)
|\\
&\leq
\int_{\mathcal{S} \setminus 2B}
 \left|K(z,y) -
 K(x,y)\right| \left|f(y)\right| \, d\nu(y)\\
 &\stackrel{\eqref{e-growth-holder}}{\leq}
C_K \int_{\mathcal{S} \setminus 2B} \frac{d(x,z)^{\beta}}{d(z,y)^{n+\beta}} |f(y)| \, d\nu(y)\\
 &\leq
C_K (324 \cdot 8^{-k})^\beta
 \int_{\mathcal{S} \setminus 2B} 
\frac{|f(y)|}{d(z,y)^{n+\beta}}  \, d\nu(y)
\stackrel{\eqref{e-tech}}{\leq}
\frac{C_K C_\nu 2^{n}}{2^\beta - 1} M_\nu f(x)
\end{align*}
where $C_K > 0$ is the constant from \eqref{e-growth-metric} and \eqref{e-growth-holder}.

Combining the above estimates gives a constant $C > 0$
independent of $U$ and $x$
such that
$$
T_{\nu,*} (\chi_{\mathcal{S} \setminus 2B_2} f)(x)
\leq 
CM_\nu f(x) + \lambda + CM_\nu f(x)
\leq \lambda+2CM_\nu f(x).
$$
Choose $\delta < \varepsilon/(4C)$.
Thus, if $T_{\nu,*}f(x) > (1 + \varepsilon) \lambda$ and $M_{\nu}f(x) \leq \delta \lambda$, then
\begin{align*}
T_{\nu,*} (\chi_{2B_2} f)(x)
\geq
T_{\nu,*}  f(x)
-
T_{\nu,*} (\chi_{\mathcal{S} \setminus 2B_2} f)(x)
\geq
(1+\varepsilon)\lambda - (1+2C\delta)\lambda
> \varepsilon \lambda / 2.
\end{align*}
\end{proof}
Let $\varepsilon>0$ and choose $\delta$ as above.
Fix an index $i \in \mathcal{I}_1$, 
write $U := U_i$, $x_0 := x_i$, and $k:=k_i$,
and define $B_1$ and $B_2$ as before.
Since $i \in \mathcal{I}_1$, it follows from \eqref{e-dubstep} that the ball $B_2$ is $b$-doubling. 
Thus our hypotheses imply that 
there is a subset $G_U \subset B_2$ such that 
$T_{\nu,*}$ 
satisfies \eqref{e:goodlamb}
on $G_U$ with a constant independent of $U$,
and
\eqref{e-dubstep} gives
$$
\nu(G_U) \geq \theta \nu(B_2) \geq \theta \nu(U).
$$
If there is some point $x_1 \in G_U$ such that $M_\nu f(x_1) \leq \delta \lambda$, 
then the previous lemma 
and
\eqref{e:goodlamb}
give
\begin{align*}
    \nu(\{x \in G_U \, : \, &T_{\nu,*}f(x) > (1+\varepsilon)\lambda, \, M_\nu f(x) \leq \delta \lambda\})\\
    &\leq 
    \nu(\{x \in G_U \, : \, T_{\nu,*}(\chi_{2B_2}f)(x) > \varepsilon \lambda / 2\})\\
    &\leq \frac{2c}{\varepsilon \lambda} \Vert \chi_{2B_2} f \Vert_{L^1(\nu)}\\
    &\leq
    \frac{2c}{\varepsilon \lambda} \int_{B(x_1,8 \cdot 8^{-k})} |f| \, d\nu
    \leq \frac{2c}{\varepsilon \lambda} \nu(8B_2) M_\nu f(x_1)
    \leq \frac{2c \delta}{\varepsilon} \nu(8B_2).
\end{align*}
(If there is no such point $x_1 \in G_U$, then this bound is trivial.)
Writing $G_{U_i}$ to denote the set associated with $U_i$ for each $i \in \mathbb{N}$,
we therefore have
\begin{align*}
    \nu(&\{x \in \mathcal{S} \, : \, T_{\nu,*}f(x) > (1+\varepsilon)\lambda, \, M_\nu f(x) \leq \delta \lambda\}) \\
    &\leq
    \sum_{i \in \mathcal{I}_2} \nu(U_i)
    +
    \sum_{i \in \mathcal{I}_1} \nu(U_i \setminus G_{U_i})\\
    & \qquad + 
    \sum_{i \in \mathcal{I}_1} \nu(\{x \in G_{U_i} \, : \, T_{\nu,*}f(x) > (1+\varepsilon)\lambda, \, M_\nu f(x) \leq \delta \lambda\})\\
    &\leq 
    \sum_{i \in \mathcal{I}_2} \nu(U_i)
    +
    \sum_{i \in \mathcal{I}_1} \nu(U_i)
    -
    \theta\sum_{i \in \mathcal{I}_1}  \nu(U_i)
    +
   \frac{2c \delta}{\varepsilon} \sum_{i \in \mathcal{I}_1}  \nu(B(x_i,16 \cdot 8^{-k_i})) \\
    &
    \leq
    \left( 1 - \tfrac{\theta}{2} + \tfrac{2c \delta C_D}{\varepsilon} \right)\nu (\Omega_\lambda).
\end{align*}
In the last inequality, we applied \eqref{e-morethanhalf} and the finite overlap of balls of the form $B(x_i,16 \cdot 8^{-k_i})$.
Choosing $\delta$ possibly smaller so that $\delta  \leq \tfrac{\varepsilon \theta}{8 c C_D}$,
we conclude that
$$
\nu(\{x \in \mathcal{S}  :  T_{\nu,*}f(x) > (1+\varepsilon)\lambda, \, M_\nu f(x) \leq \delta \lambda\})
\leq 
(1 - \tfrac{\theta}{4}) \nu(\Omega_\lambda).
$$
\end{proof}

We now prove Theorem \ref{t:goodlambda}.

\begin{proof}[Proof of Theorem \ref{t:goodlambda}]
We will use 
Lemma~\ref{l:goodlambda}
to verify that $T_{\nu,*}$ is bounded in $L^p(\nu)$
in the standard manner.
Suppose $f \in L^p(\nu)$ has compact support.
To avoid problems that arise if 
$\Vert T_{\nu,*} f \Vert_{L^p(\nu)} = \infty$,
we will truncate $T_{\nu,*} f$.
For each $m \in \mathbb{N}$, 
set $g_m = \min \{ m , T_{\nu,*}f \}$.
Since $\supp(f)$ is compact, $\nu( \{ x \in M \, : \, d(x,\supp(f)) < 1 \} ) < \infty$.
This fact together with Lemma~\ref{l-supp} below implies that $g_m \in L^p(\nu)$.

We will now prove that $\Vert g_m \Vert_{L^p(\nu)}^p \leq C' \Vert f \Vert_{L^p(\nu)}^p$ for a constant $C'$ independent of $m$.
Passing $m \to \infty$ will then complete the proof that $T_{\nu,*}$ is bounded in $L^p(\nu)$.
Fix $m \in \mathbb{N}$ and $\varepsilon,\lambda > 0$. 
Again, write $\mathcal{S} := \supp(\nu)$.
If $\lambda \geq m$, then
$$
\{ x \in \mathcal{S} : g_m(x) > (1+\varepsilon)\lambda\} = \emptyset.
$$
If $\lambda < m$, then 
$\{ T_{\nu,*} f > \lambda \} = \{g_m > \lambda\}$, so \eqref{e:lambda} gives
\begin{align*}
\nu( \{ x \in \mathcal{S} &: g_m(x) >(1+\varepsilon)\lambda \} )\\
    &\leq 
    \nu( \{ x \in \mathcal{S} : T_{\nu,*}f(x) >(1+\varepsilon)\lambda \} )\\
    &\leq
    \nu(\{x \in \mathcal{S} : T_{\nu,*}f(x) > (1+\varepsilon)\lambda, \, M_\nu f(x) \leq \delta \lambda\}) \\
    & \qquad +
    \nu(\{x \in \mathcal{S} : M_\nu f(x) > \delta \lambda\})
    \\
&\leq 
\eta \nu( \{ x \in \mathcal{S} : T_{\nu,*}f(x) > \lambda \} )
+
    \nu(\{x \in \mathcal{S} : M_\nu f(x) > \delta \lambda\})
    \\
&\leq 
\eta \nu( \{ x \in \mathcal{S} : g_m(x) > \lambda \} )
+
    \nu(\{x \in \mathcal{S} : M_\nu f(x) > \delta \lambda\}).
\end{align*}

Multiplying this inequality by $p \lambda^{p-1}$ and integrating $\lambda$ from 0 to $\infty$ gives
$$
(1+\varepsilon)^{-p} \Vert g_m \Vert_{L^p(\nu)}^p
\leq
\eta \Vert g_m \Vert_{L^p(\nu)}^p
+
\delta^{-p} \Vert M_{\nu} f \Vert_{L^p(\nu)}^p,
$$
and choosing $\varepsilon$ small enough allows us to conclude that
$$
\Vert g_m \Vert_{L^p(\nu)}^p
\leq
C \Vert M_{\nu} f \Vert_{L^p(\nu)}^p
\leq
C' \Vert f \Vert_{L^p(\nu)}^p
$$
for constants $C,C'>0$ independent of $m$.
Passing $m \to \infty$ completes the proof
of Theorem~\ref{t:goodlambda}.
\end{proof}

It remains to prove the following standard lemma.

\begin{lemma}
\label{l-supp}
If $\nu$ is an upper $n$-regular measure on a metric space $M$, 
and if $f \in L^p(\nu)$ for some $1<p<\infty$ has compact support, then
$$
\int_{ \{x \, : \, d(x,\supp(f)) \geq 1 \} } \left| T_{\nu,*}f(x)  \right|^p \, d\nu(x) < \infty.
$$
\end{lemma}

\begin{proof}
Write $Z := \supp(f)$.
For each $x \in M$ with $d(x,Z) \geq 1$,
set $R_x := d(x,Z)/2$ and $B := B(x,R_x)$ so that $f \equiv 0$ on $B$.
Then H\"{o}lder's inequality gives
\begin{align*}
    T_{\nu,*}f(x) 
    \leq
    \int_{M \setminus B} |K(x,y)||f(y)| \, d\nu(y) 
    &\leq
    C_K\int_{M \setminus B} \frac{|f(y)|}{d(x,y)^{n}} \, d\nu(y)\\
    &\leq
    C_KR_x^{-n} \Vert f \Vert_{L^p(\nu)} \nu(Z)^{(p-1)/p}.
\end{align*}
For each integer $k$, write 
$$
A_k := \{ x \in M \, : \, 2^k \leq d(x,Z) < 2^{k+1} \}.
$$
Note that 
$A_k$ is contained in a ball of radius $2^{k+1} + \diam(Z)$.
Thus
\begin{align*}
    C_K^{-1} \Vert f \Vert_{L^p(\nu)}^{-p} \nu(Z)^{-(p-1)} 
    \int_{ \{x \, : \, d(x,Z) \geq 1 \} } &\left| T_{\nu,*}f(x)  \right|^p \, d\nu(x)\\
    &\leq
    \int_{ \{x \, : \, d(x,Z) \geq 1 \} } \frac{2^n}{d(x,Z)^{np}} \, d\nu(x)\\
    &= 
    \sum_{k=0}^\infty 
    \int_{A_k} \frac{2^n}{d(x,Z)^{np}} \, d\nu(x)\\
    &\leq
    2^nC_\nu \sum_{k=0}^\infty 
    \frac{2^{nk+n} + \diam(Z)^n}{2^{knp}}
    <\infty.
\end{align*}
\end{proof}

\bibliography{SIObib} 
\bibliographystyle{acm}

\end{document}